\documentclass[12pt,a4paper]{article}
\usepackage{amsmath}
\usepackage{amsthm}
\usepackage{amsfonts}
\usepackage{amssymb}
\usepackage{latexsym}

\DeclareFontFamily{U}{futm}{}
\DeclareFontShape{U}{futm}{m}{n}{  <-> s * [.92] fourier-bb}{}
\DeclareSymbolFont{Ufutm}{U}{futm}{m}{n}
\DeclareSymbolFontAlphabet{\mathbb}{Ufutm}

\newtheorem{theorem}{Theorem}
\newtheorem{lemma}{Lemma}
\newtheorem{corollary}{Corollary}

\newtheorem{remark}{Remark}
\newtheorem{algorithm}{Algorithm}

\theoremstyle{definition}
\newtheorem{definition}{Definition}

\newtheorem{example}{Example} % The '*' makes it unnumbered

\begin{document}

\title{\Large{\textbf{Quasigroup based crypto-algorithms}}}
\author{\normalsize {Victor Shcherbacov}}

 \maketitle

\begin{abstract}
\noindent
 Modifications of Markovski quasigroup based crypto-algorithm have been  proposed. Some of these modifications are based on the systems of orthogonal $n$-ary groupoids.  $T$-quasigroups based stream ciphers have been constructed.

\medskip


\medskip

\end{abstract}

\tableofcontents

\bigskip

\section{Introduction}

\subsection{Preliminaries}

This paper is an extended variant and a prolongation of the paper \cite{PS_VS_11}.
Information on quasigroups and $n$-ary quasigroups it is possible to find in \cite{VD, 2, 1a, HOP}, on ciphers in \cite{MENEZES, LM}. Some applications of quasigroups in cryptology are described in  \cite{DK1, DK2, OS, Grosek_Sys, SCERB_09_CSJM}.

Two main elementary methods of ciphering the information are known.

(i). Symbols in a plaintext (or in its piece (its bit))  are permuted by some law. One of the first known ciphers of
such kind is cipher "Scital" (Sparta, 2500 years ago).

(ii). All symbols in a fixed alphabet are changed by a law on other letters of this alphabet. One of the first
ciphers of such kind was Cezar's cipher ($x\rightarrow x+3$ for any letter of Latin alphabet, for example $a\rightarrow d, b \rightarrow e$ and so on).

In many contemporary ciphers (DES, old Russian GOST, Blowfish \cite{NM, DPP}) the methods (i) and (ii) are used with some modifications.
Therefore, permutations and substitutions are main elementary cryptographical procedures.

What does  the use of quasigroups in cryptography give us? It gives the same permutations and substitutions but easy generated, requiring not very big volume of a device memory, acting "locally" on only one block of a plain-text.

"Stream ciphers are an important class of encryption algorithms. They encrypt individual characters (usually binary digits) of a plaintext message one at a time, using an encryption transformation which varies with time.

By contrast, block ciphers  tend to simultaneously encrypt groups of characters of a plaintext message using a fixed encryption transformation. Stream ciphers are generally faster than block ciphers in hardware, and have less complex hardware circuitry.

They are also more appropriate, and in some cases mandatory (e.g., in some telecommunications applications), when buffering is limited or when characters must be individually processed as they are received. Because they
have limited or no error propagation, stream ciphers may also be advantageous in situations where transmission errors are highly probable" \cite{MENEZES}.

Stream-ciphers based on quasigroups and their parastrophes were  discovered in the end of the XX-th century \cite{MARKOVSKI, MGB, MARKOV_00}.

Often by enciphering a block (a letter) $B_i$ of a plaintext the previous ciphered block $C_{i-1}$ is used. Notice that Horst Feistel was one of the first who proposed such method of encryption (Feistel net)   \cite{Feistel_73}.

It is clear that by the construction of a stream cipher it is impossible to use method (i) (see above). But it is possible to use method (ii) and Feistel schema. Of course these methods cannot be unique.

\subsection{Basic definitions}

We give some definitions.
A sequence $x_m, x_{m+1}, \dots, x_n$, where $ m, n$ are natural numbers and $m\leq n$,  will be denoted by
$x_m^n$. If $m > n$,  then $x_m^n$ will be considered empty. The sequence $x, \dots, x$ (k times) will be denoted by $\overline{x}^k$. The expression $\overline{1,
n}$ designates the set $\{1, 2, \dots, n\}$ of natural numbers \cite{2}.

A non-empty set $Q$  together with an $n$-ary operation $A : Q^n \rightarrow Q$, $n\geq 2$
is called $n$-groupoid and it is denoted by $(Q, A)$.

It is convenient to define $n$-ary quasigroup in the following manner.
\begin{definition} \label{def2} An $n$-ary groupoid $(Q, A)$ with $n$-ary operation $A$ such
that in the equality $A(x_1,$ $ x_2, \dots, x_n) = x_{n+1}$ the knowledge of any $n$  elements from the elements $x_1, x_2, \dots,$
$ x_n, x_{n+1}$ uniquely specifies the remaining one is called $n$-ary quasigroup  \cite{2}.
\end{definition}

From Definition \ref{def2}  follows \cite{VD, HOP, SCERB_03} that any quasigroup $(Q, A)$ defines else $((n+1)! -1)$ $n$-quasigroups,
so-called parastrophes of quasigroup $(Q,A)$.

In binary case any quasigroup $(Q, A)$ defines else five quasigroups namely  $(Q,{}^{(13)}A)$, $(Q,{}^{(23)}A)$, $(Q,{}^{(12)}A)$, $(Q,{}^{(123)}A)$, $(Q,{}^{(132)}A)$. See \cite{VD, HOP, SCERB_07} for details.

We give  classical equational definition of binary quasigroup \cite{EVANS_49}.

\begin{definition}  \label{EQUT_QUAS_DEF}
 A binary groupoid $(Q, A)$ is called a binary quasigroup if on the set $Q$ there exist
operations ${}^{(13)}A$ and ${}^{(23)}A$ such that in the algebra $(Q, A, {}^{(13)}A, {}^{(23)}A)$ the following
identities are fulfilled:
\begin{equation}
A({}^{(13)}A(x, y), y) = x, \label{(2e)}
\end{equation}
\begin{equation}
{}^{(13)}A(A(x, y), y) = x, \label{(4e)}
\end{equation}
\begin{equation}
A(x, {}^{(23)}A (x, y)) = y, \label{(1e)}
\end{equation}
\begin{equation}
{}^{(23)}A(x, A (x, y)) = y. \label{(3e)}
\end{equation}
\end{definition}

By tradition the operation $A$ is denoted by $\cdot$, ${}^{(23)}A$ by $\backslash$ and ${}^{(13)}A$ by $\slash$.

It is possible to give equational definition of  $n$-ary quasigroup as a generalization of Definition  \ref{EQUT_QUAS_DEF}. We follow \cite{2, Petrescu_10}.

\begin{definition}
 An $n$-ary groupoid $(Q, A)$ is called an $n$-ary quasigroup if on the set $Q$ there exist
operations ${}^{(1,\, n+1)}A$,  ${}^{(2,\, n+1)}A$, $\dots$, ${}^{(n,\, n+1)}A$ such that in the algebra $(Q, A, {}^{(1,\,n+1)}A, \dots, $ $ {}^{(n,\,n+1)}A)$  the following identities are fulfilled for all $i \in \overline{1,n}$:
\begin{equation}
A(x_1^{i-1}, {}^{(i,\, n+1)}A(x_1^n), x_{i+1}^n) = x_i, \label{(1ne)}
\end{equation}
\begin{equation}
{}^{(i,\, n+1)}A(x_1^{i-1}, A(x_1^n), x_{i+1}^n) = x_i. \label{(2ne)}
\end{equation}
\end{definition}

In \cite{GLUKH_89} it is proved that any $n$-ary quasigroup of order $k\geq 7$ is a special kind composition  of binary quasigroups isotopic to a fixed quasigroup.\footnote{The author thanks Prof.~F.M.~Sokhatsky that informed his  about this result of M.M.~Glukhov.}

\begin{definition} Let $(G,\cdot)$ be a groupoid and let $a$ be a fixed element in $G$.
Translation maps $L_a$ (left) and $R_a$  (right) are defined by the following equalities  $L_a x = a\cdot x$, $R_a x = x\cdot a$ for
all $x\in G$. For  quasigroups it is possible to define  a third kind of translation, namely, middle translations. If $P_a$ is a  middle
translation of a quasigroup $(Q,\cdot)$, then  $x\cdot P_{a}x = a$ for all $x\in Q$ \cite{BELAS}.
\end{definition}

It is well known that in a quasigroup $(Q, \cdot)$ any  left and right translation is  a bijective map of the set $Q$ \cite{VD, HOP}.

\subsection{Quasigroup based cryptosystem} \label{Quasigroup based cryptosystem}

We give based on binary quasigroup encoding algorithm. We use \cite{SCERB_03}.

A quasigroup $(Q,\cdot)$ and its $(23)$-parastrophe $(Q,\backslash)$  satisfy the following identities   $ x\cdot (x\backslash y) = y$,
$ x\backslash (x\cdot y) = y$. These are identities (\ref{(1e)}) and (\ref{(3e)}), respectively.

The authors \cite{MARKOVSKI, MGB} propose to use this quasigroup  property  to construct the following stream cipher.

\begin{algorithm} \label{ALG1} Let $Q$ be a non-empty finite  alphabet,  $k$ be a natural number, $u_i,
v_i \in Q$, $i\in \{1,..., k\}$.  Define a quasigroup $(Q, A)$.
It is clear that the quasigroup $(Q, {}^{(23)}A)$ is defined in a unique way.

Take a fixed element $l$ ($l\in Q$), which  is called a leader.

Let $u_1 u_2... u_k$ be a $k$-tuple  of letters from $Q$.

It is proposed  the following ciphering procedure

$v_1 = A(l, u_1)$,

$v_{i}= A(v_{i-1}, u_{i})$, $i= 2,..., k$.

  Therefore we obtain the following cipher-text $v_1v_2 \dots v_k$.

  The deciphering algorithm is constructed in the following way: $u_1= {}^{(23)}A(l, v_1)$,
$u_{i}= {}^{(23)}A(v_{i-1}, v_{i}),$ $i = 2,..., k.$

Indeed ${}^{(23)}A(v_{i-1}, v_{i}) = {}^{(23)}A(v_{i-1}, A(v_{i-1}, u_{i})) \overset{(\ref{(3e)})}{=} u_i.$
\end{algorithm}

Notice, the equality $A = {}^{(23)}A$ is fulfilled if and only if  $A(x, A(x,y)) = y$ for all $x, y \in Q$.

\subsection{Modifications and generalizations} \label{GENERALIZAT_LITTLE_BIT}

The improvements and researches of Algorithm \ref{ALG1} were carried out intensively. Some information on this process is given in \cite{SCERB_03}. We thank our colleagues A.~Krapez, V.~Bakeva, V.~Dimitrova and A.~Popovska-Mitrovikj for the following new information.
\begin{remark}
 In article \cite{BAKEVA_DIM}, the authors find the distribution of $k$-tuples of letters after $n$ applications of quasigroup transformation ($k > n$) (i.e. Algorithm \ref{ALG1}) and give an algorithm for statistical attack in order to discover the original message. Also, they give some conclusions on how to protect the original messages.

 In work \cite{KRAPEZ_10}, Krapez defines parastrophic quasigroup transformation. In \cite{BAKEVA_DIM_POP}, the authors propose a modification of this transformation and give a new classification of quasigroups of order 4. Finally, in [17] the authors presented this
 transformation and gave relationship between the new classification and  the  symmetries of quasigroups.
\end{remark}

Notice, parastrophic transformations from \cite{KRAPEZ_10, DIM_BAK_POP_KR} are  promising for further applications and researches.

In Algorithm \ref{ALG1}  it is possible  to use also a quasigroup $(Q, A)$ and its  $(13)$~-, $(123)$-~, $(132)$-parastrophe since   quasigroup $(Q, A)$ and these  parastrophes fulfill the following identities, namely, identities  (\ref{(4e)}),  (\ref{(3ec)}), and (\ref{(4ec)}),  respectively \cite{SCERB_07, KRAPEZ_10, DIM_BAK_POP_KR}.
\begin{equation}
{}^{(123)}A(A (x, y), x) = y \label{(3ec)}
\end{equation}
\begin{equation}
{}^{(132)}A(y, A(x, y)) = x \label{(4ec)}
\end{equation}

More details in this direction are in \cite{KRAPEZ_10}.

\iffalse

\begin{example}
 Let alphabet $A$ consists from the letters $a, b, c$. Take the quasigroup $(A, \cdot)$:
\[
{\begin{array}{c|ccc}
\cdot & a & b & c  \\
\hline
a & b & c & a   \\
b & c & a & b   \\
c & a & b & c
\end{array}}
\]
Then  $(A, \backslash)$ has the following Cayley table
\[
{\begin{array}{c|ccc}
\backslash & a & b & c  \\
\hline
a & c & a & b   \\
b & b & c & a   \\
c & a & b & c
\end{array}}
\]

Let $l=a$ and open text is $u = b\,b\,c\,a\,a\,c\,b\,a$. Then the cipher text  is $v = c\,b\,b\,c\,a\,a\,c\,a$.
Applying  decoding function  we get $b\,b\,c\,a\,a\,c\,b\,a=u$.
\end{example}
\fi

In \cite{MGB},  the authors claimed  that this cipher is resistant  to the brute force attack (exhaustive search) and to  the
statistical attack (in many languages some letters meet more frequently, than other letters)\footnote{The author thanks his  colleagues A.~Krapez, V.~Bakeva, V.~Dimitrova and A.~Popovska-Mitrovikj for this information (private letter).}.
Later similar results were presented in \cite{OS}.

In dissertation of Milan Vojvoda \cite{VOIVODA_04} has been  proved that this cipher is not resistant to chosen ciphertext attack  and chosen plaintext attack.
It is claimed that this cipher is not resistant to   special kind of statistical attack (Slovak language) \cite{VOIVODA_04}.

There exist a  few other ways to generalize Algorithm \ref{ALG1}.
The most obvious way is to increase arity of a quasigroup, i.e. instead of binary to apply $n$-ary ($n\geq 3$) quasigroups. This way was proposed in \cite{SCERB_03, SCERB_03_1} and  was realized in \cite{Petrescu_07, Petrescu_10}. See below Algorithm \ref{ALGn_ar}.
Notice Prof. A.~Petrescu writes that he found this $n$-ary generalization independently.

In \cite{PS_VS_11}, the authors proved that cipher based on Algorithm \ref{ALGn_ar}  is not resistant to chosen ciphertext attack  and chosen plaintext attack.

Some modifications in order to make Algorithm \ref{ALG1} more resistant against known attacks  can be found in  \cite{KRAPEZ_10, DIM_BAK_POP_KR}. One of these attempts, taking into consideration Vojvoda results \cite{VOIVODA_04},   was proposed in \cite{SCERB_09_CSJM}. Namely instead of a binary quasigroup and its parastrophe  it  was proposed to use a system of $n$ $n$-ary orthogonal operations (groupoids).

Also it was proposed to use these two crypto-primitives  together in one cryptographical procedure.

\subsection{A modification of Algorithm \ref{ALG1}}

Sometimes only the use of other record of a mathematical fact leads to a generalization.

We re-write  Algorithm \ref{ALG1} using concept of translation in the following way:
\begin{algorithm} \label{ALG1_TRANSL} Let $Q$ be a non-empty finite  alphabet.  Define a quasigroup $(Q, \cdot)$.
It is clear that the quasigroup $(Q, \overset{(23)}{\cdot} )$ is defined in a unique way.

Take a fixed element $l$ ($l\in Q$), which  is called a leader.

Let $u_1 u_2... u_k$ be a $k$-tuple  of letters from $Q$.

It is proposed  the following ciphering procedure

$v_1 = l\cdot u_1 = L_l u_1$,

$v_{2}= v_{1} \cdot u_{2} = L_{v_{1}}u_{2}$.

$v_{i}= v_{i-1} \cdot u_{i} = L_{v_{i-1}}u_{i}$, $i = 3,..., k$.

Therefore we obtain the following cipher-text $v_1v_2 \dots v_k$.

The deciphering algorithm is constructed in the following way. We have the following cipher-text: $v_1v_2 \dots v_k$.  Recall $L_{a}^{\overset{(23)}{\cdot}} = (L_a^{\cdot})^{-1}$ for any $a \in Q$ \cite{SCERB_03}. Below we shall denote translation $L_{a}^{\overset{(23)}{\cdot}}$ as  $L_a^{\ast}$, translation  $L_{a}^{\cdot}$ as $L_{a}$ for any $a\in Q$.
 Then
 \begin{equation}
\begin{split}
 & u_1= l \overset{(23)}{\cdot} v_1 =  {L^{\ast}_l}\left( v_1 \right) =
{L^{\ast}_l} \left( L_l u_1 \right) = \\ & L_{l}^{-1} \left( L_l u_1 \right) = u_1; \\
&u_{i} = v_{i-1} \overset{(23)}{\cdot}  v_{i} =
{L^{\ast}_{v_{i-1}}} \left(v_{i}\right) =
{L^{\ast}_{v_{i-1}}} \left( L_{v_{i-1}}u_{i}\right) = \\ & L_{v_{i-1}}^{-1}\left( L_{v_{i-1}}u_{i}\right)= u_{i}
\end{split}
\end{equation}
for all $i \in \overline{2, k}$.
\end{algorithm}

From this form  of Algorithm \ref{ALG1} we  can obtain easily  the following generalization. Instead of translations $L_x$, $x\in Q$, we propose to use in the enciphering part of this algorithm powers of these translations, i.e., to use   permutations of the form $L^k_x$, $k\in \mathbb{Z}$,  instead of  permutations of the form $L_x$.

The proposed modification  forces us to use permutations of the form $L^k_x$, $k\in \mathbb{Z}$, also in the decryption procedure.

\begin{algorithm} \label{ALG1_POWER_OF_TRANSL} Let $Q$ be a non-empty finite  alphabet.  Define a quasigroup $(Q, \cdot)$.
It is clear that the quasigroup $(Q, \overset{(23)}{\cdot} )$ is defined in a unique way.

Take a fixed element $l$ ($l\in Q$), which  is called a leader.

Let $u_1 u_2... u_k$ be a $k$-tuple  of letters from $Q$.

It is proposed  the following ciphering procedure
  \begin{equation} \label{ALG_4_equations}
\begin{split}
& v_1 =  L^a_l u_1, a\in \mathbb{Z},\\
& v_{2}=  L^b_{v_{1}}u_{2},  b\in \mathbb{Z},\\
& v_{i}=  L^c_{v_{i-1}}u_{i}, i \in \overline{3, k}, c\in \mathbb{Z}.
\end{split}
\end{equation}
Therefore we obtain the following cipher-text $v_1v_2 \dots v_k$.
 The deciphering algorithm is constructed in the following way. We use notations of Algorithm \ref{ALG1_TRANSL}. Recall $({L^{\ast}_x})^{a} = L^{-a}_x$ for all $x\in Q$.  Then
  \begin{equation}
\begin{split}
 & ({L^{\ast}_l})^{a}\left( v_1 \right) = ({L^{\ast}_l})^{a}\left(L^a_l u_1 \right) = u_1,
 \\
 &({L^{\ast}_{v_1}})^{b}\left( v_2 \right) = ({L^{\ast}_{v_1}})^{b}\left(L^b_{v_1} u_2 \right) = u_2,  \\
 & ({L^{\ast}_{v_{i-1}}})^{c}\left( v_i \right) = ({L^{\ast}_{v_{i-1}}})^{c}(L^c_{v_{i-1}}u_{i})  = u_i, i \in \overline{3, k}.
\end{split}
\end{equation}
\end{algorithm}
Notice, the elements $a, b, c$ in equalities (\ref{ALG_4_equations}) should be vary from step to step in order to protect this Algorithm against  chosen  plain-text and chosen cipher-text attack. It is clear that the right and middle \cite{SCERB_03} translations are also possible to use   in Algorithm \ref{ALG1_POWER_OF_TRANSL} instead of the left translations. See below.

\subsection{$n$-ary analogs of binary algorithms}

We give $n$-ary analog of Algorithm \ref{ALG1} \cite{Petrescu_07, PS_VS_11}.

\begin{algorithm} \label{ALGn_ar} Let $Q$ be a non-empty finite alphabet,  $k$ be a natural number, $u_i,
v_i \in Q$, $i\in \{1,..., k\}$.  Define an $n$-ary quasigroup $(Q, f)$.
It is clear that any quasigroup $(Q, {}^{(i,\, n+1)}f)$ for any fixed value $i$ is defined in a unique way. Below for simplicity we put $i=n$.

Take fixed elements $l_1^{(n-1)(n-1)}$ ($l_i\in Q$), which  are called  leaders.

Let $u_1 u_2... u_k$ be a $k$-tuple  of letters from $Q$.

It is proposed  the following ciphering (encryption) procedure
\begin{equation} \label{seven}
\begin{split}
& v_1 = f(l_1^{n-1}, u_1),  \\
& v_2 = f(l_n^{2n-2}, u_2), \\
& \dots , \\
& v_{n-1} = f(l_{n^2 - 3n +3}^{(n-1)(n-1)}, u_{n-1}), \\
& v_{n}= f(v_{1}^{n-1}, u_{n}), \\
& v_{n+1}= f(v_{2}^{n}, u_{n+1}), \\
& v_{n+2}= f(v_{3}^{n+1}, u_{n+2}), \\
& \dots
\end{split}
\end{equation}
Therefore we obtain the following cipher-text $v_1v_2 \dots, v_{n-1}, v_{n}, v_{n+1}, \dots$.

  The deciphering algorithm also  is constructed similarly with binary case:
  \begin{equation} \label{eight}
\begin{split}
& u_1= {}^{(n,\, n+1)}f(l_1^{n-1}, v_1),  \\
& u_{2}= {}^{(n,\, n+1)}f(l_n^{2n-2}, v_{2}), \\
& \dots , \\
& u_{n-1} = {}^{(n,\, n+1)}f(l_{n^2 - 3n +3}^{(n-1)(n-1)}, v_{n-1})\\
& u_{n}= {}^{(n,\, n+1)}f(v_{1}^{n-1}, v_{n}),\\
& u_{n+1}= {}^{(n,\, n+1)}f(v_{2}^{n}, v_{n+1}), \\
& u_{n+2}= {}^{(n,\, n+1)}f(v_{3}^{n+1}, v_{n+2}), \\
& \dots
\end{split}
\end{equation}
Indeed, for example,  ${}^{(n,\, n+1)}f(v_{1}^{n-1}, v_{n}) = {}^{(n,\, n+1)}f(v_{1}^{n-1}, f(v_{1}^{n-1}, u_{n})) \overset{(\ref{(2ne)})}{=} u_n.$
\end{algorithm}

\begin{remark}
It is easy to see that in encryption procedure  (equalities  (\ref{seven})) and, therefore, in decryption procedure (equalities (\ref{eight})) it is possible to use more than one fixed $n$-quasigroup operation $f$.
\end{remark}

Below we shall denote this encryption algorithm as $G(u)$, because on any step it is enciphered only one element of a plaintext.
Probably  it makes sense to use in Algorithm  \ref{ALGn_ar}  irreducible 3-ary or 4-ary finite quasigroup \cite{2, BORISENKO_V_V, AKIV_GOLD_00, AKIV_GOLD_01}.
We give an example of 3-ary irreducible  quasigroup $(Q, A)$ of order 4 \cite[p. 115]{2}.
\begin{example}\label{TERN_NO_REDUC_QUS}
 \[
\begin{array}{cccc}
\begin{array}{c|cccc}
A_0 & 0 & 1 & 2 & 3 \\
\hline
0  & 0 & 1 & 2 & 3 \\
1  & 1 & 2 & 3 & 0 \\
2  & 2 & 3 & 0 & 1 \\
3  & 3 & 0 & 1 & 2 \\
\end{array}
&
\begin{array}{c|cccc}
A_1 & 0 & 1 & 2 & 3 \\
\hline
0  & 1 & 0 & 3 & 2 \\
1  & 0 & 1 & 2 & 3 \\
2  & 3 & 2 & 1 & 0 \\
3  & 2 & 3 & 0 & 1 \\
\end{array}
&
\begin{array}{c|cccc}
A_2 & 0 & 1 & 2 & 3 \\
\hline
0  & 2 & 3 & 0 & 1 \\
1  & 3 & 0 & 1 & 2 \\
2  & 0 & 1 & 2 & 3 \\
3  & 1 & 2 & 3 & 0 \\
\end{array}
&
\begin{array}{c|cccc}
A_3 & 0 & 1 & 2 & 3 \\
\hline
0  & 3 & 2 & 1 & 0 \\
1  & 2 & 3 & 0 & 1 \\
2  & 1 & 0 & 3 & 2 \\
3  & 0 & 1 & 2 & 3 \\
\end{array}
\end{array}
\]
Notice $A(0, 1, 2) = A_0(1, 2) = 3,$ $A(2, 3, 2) = A_2(3, 2) = 3.$ Moreover $A(0, 1, x) = A(2, 3, x)$ for any  $x\in Q$. Then translations $T(0, 1, -)$ and $T(2, 3, -)$ are equal, pairs of leaders $(0, 1)$ and $(2, 3)$ are equal from cryptographical point of view.
\end{example}

Recall there exist two groups of order 4, namely cyclic group $Z_4$  and Klein group $Z_2\times Z_2$. Any binary quasigroup of order 4 is a group isotope \cite{A1,A2}.
\begin{lemma}
      Quasigroup from Example \ref{TERN_NO_REDUC_QUS} is not an isotope of a $3$-ary group $(Q, f)$ with the form $f(x_1^3) = x_1+x_2+x_3$ where $(Q, +)$ is a binary group of order 4.
\end{lemma}
\begin{proof}
If a quasigroup is an isotope  of a $3$-ary group $(Q, f)$ with the form $f(x_1^3) = x_1+x_2+x_3$ where $(Q, +)$ is a binary group, then this quasigroup is reducible \cite[Corollary, p. 115]{2}.
\end{proof}

A translation of $n$-ary quasigroup $(Q, f)$ ($n>2$) will be denoted as $ T(a_1, \dots, $ $a_{i-1}, -,$
$a_{i+1}, \dots, a_n)$, where $a_i\in Q$ for all $i\in \overline {1,n}$ and
\begin{equation*}
T(a_1, \dots, a_{i-1}, -, a_{i+1}, \dots, a_n) x = f(a_1, \dots, a_{i-1}, x, a_{i+1}, \dots, a_n)
\end{equation*}
for all $x\in Q$.

 From definition of $n$-ary quasigroup  follows that any translation of $n$-ary quasigroup
 $(Q,f)$ is a permutation of the set $Q$.

\begin{lemma}\label{INVERSE_TRANS_IN_N_PAR}
If ${}_fT(a_1, \dots, a_{n-1}, -)$ is a translation of a quasigroup $(Q, f)$, then
$${}_fT^{-1}(a_1, \dots, a_{n-1}, -) = {}_{{}^{(n, n+1)}f} T(a_1, \dots, a_{n-1}, -)$$
\end{lemma}
\begin{proof}
In the proof we omit the symbol $f$ in the notation of translations of  quasigroup $(Q, f)$.
We have
\begin{equation}
\begin{split}
& T^{-1} (a_1, \dots, a_{n-1}, -) (T(a_1, \dots, a_{n-1}, -)x) = \\ &
T^{-1} (a_1, \dots, a_{n-1}, -) f(a_1, \dots,  a_{n-1}, x) = \\ &
{{}^{(n, n+1)}f} (a_1, \dots, a_{n-1}, f(a_1, \dots,  a_{n-1}, x)) \overset{ (\ref{(2ne)})}{=}  x
\end{split}
\end{equation}
\end{proof}

We propose an  $n$-ary analogue of Algorithm \ref{ALG1_POWER_OF_TRANSL}.

\begin{algorithm} \label{ALG1_POWER_OF_TRANSL_N_ARY} Let $Q$ be a non-empty finite  alphabet.  Define an $n$-ary quasigroup $(Q, f)$.
It is clear that the quasigroup $(Q, {}^{(n, n+1)}f )$ is defined in a unique way.

Take fixed elements $l_1^{(n-1)(n-1)}$ ($l_i\in Q$), which  are called  leaders.

Let $u_1 u_2... u_k$ be a $k$-tuple  of letters from $Q$.

It is proposed  the following ciphering (encryption) procedure

  \begin{equation} \label{ALG_4_equations_N_ARY_TR}
\begin{split}
& v_1 = T^a(l_1, l_2, \dots, l_{n-1}, u_1),  \\
& v_2 = T^b(l_n, l_{n+1}, \dots, l_{2n-2}, u_2), \\
& \dots , \\
& v_{n-1} = T^c(l_{n^2 - 3n +3}, \dots, l_{(n-1)(n-1)}, u_{n-1}), \\
& v_{n}= T^d(v_{1}, \dots, v_{n-1}, u_{n}), \\
& v_{n+1}= T^e(v_{2}, \dots, v_{n}, u_{n+1}), \\
& v_{n+2}= T^t(v_{3}, \dots, v_{n+1}, u_{n+2}), \\
& \dots
\end{split}
\end{equation}
Therefore we obtain the following cipher-text $v_1v_2 \dots v_k$.

 Taking into consideration Lemma \ref{INVERSE_TRANS_IN_N_PAR} we can say that  deciphering algorithm is possible, it is constructed similarly with the deciphering in  Algorithm \ref{ALG1_POWER_OF_TRANSL}.
\end{algorithm}

\begin{remark}
It is easy to see that in Algorithm \ref{ALG1_POWER_OF_TRANSL_N_ARY} it is possible to use  various quasigroup translations and to take  quasigroups of various arity.
\end{remark}

\section{Ciphers based on orthogonal $n$-ary groupoids}

\subsection{Some definitions}

We give classical definition of orthogonality of  $n$-ary operations \cite{YAC, BEL_YAK}.
\begin{definition} \label{n_ORTHOGON}
 $n$-ary groupoids  $(Q, f_1)$, $(Q, f_2)$, $\dots$, $(Q, f_n)$ are called orthogonal, if for any fixed
$n$-tuple $a_1, a_2, \dots, a_n$ the following system of equations

\begin{equation} \label{SYSTEM_16}
\left\{
\begin{split}
& f_1(x_1, x_2, \dots , x_n) = a_1\\
& f_2(x_1, x_2, \dots , x_n) = a_2 \\
& \dots \\
& f_n(x_1, x_2, \dots , x_n) = a_n
\end{split}
\right.
\end{equation}
 has a unique solution.
\end{definition}

If the set $Q$ is  finite, then any system of $n$ orthogonal $n$-ary groupoids $(Q, f_i)$ $i\in \overline{1, n}$, defines a permutation of the set $Q^n$ and vice versa \cite{SYSTEMS_ORTHOGON, BEL_YAK, YAC}. Therefore if  $|Q| = q$,  then  there exist $(q^n)!$  systems of  $n$-ary orthogonal groupoids defined on the set $Q$.

\iffalse
Definition~\ref{n_ORTHOGON} it is possible to use in the case when the set $Q$ is infinite.
\begin{example} \label{SYST_EQ_3}
Operations $A_1(x_1, x_2, x_3) = 1\cdot x_1 + 0\cdot x_2 + 0\cdot x_3$, $A_2(x_1, x_2, $ $x_3) = 0\cdot x_1 +
1\cdot x_2 + 0\cdot x_3$, $A_3(x_1, x_2, x_3) = 0\cdot x_1 + 0\cdot x_2 + 1\cdot x_3$ defined over the  field $R$ of real
numbers (or over a finite field) are orthogonal, since the system
$$\left\{ \begin{array}{l}
1\cdot x_1 + 0\cdot  x_2 + 0\cdot x_3 = a_1\\
0 \cdot x_1 + 1\cdot x_2 + 0\cdot x_3 = a_2 \\
0\cdot x_1 + 0\cdot x_2 + 1\cdot x_3 = a_3
 \end{array} \right.
 $$
has a unique solution for any fixed $3$-tuple $(a_1, a_2, a_3)\in R^3$.
\end{example}
\fi

There exist various generalizations of definition of orthogonality of $n$-ary operations. Fresh generalizations are in \cite{SOKH_FRYZ, SOKH_FRYZ_LOOPS}.

\begin{definition} \label{K_ORTHOG}
 $n$-ary groupoids  $(Q, f_1)$, $(Q, f_2)$, $\dots$, $(Q, f_k)$ ($2 \leq k\leq n$)  given on a set $Q$ of order $m$ are called orthogonal if the system of equations (\ref{SYSTEM_16})
 has exactly $m^{n-k}$ solutions for any k-tuple  $a_1, a_2, \dots, a_k$, where $a_1, a_2, \dots, a_k \in Q$ (see \cite{Bel_05}).
\end{definition}

If $k=n$, then from Definition \ref{K_ORTHOG} we obtain standard Definition \ref{n_ORTHOGON}.
Definition of orthogonality of binary systems has rich and  long history \cite{DK1}. About $n$-ary case, for example, see \cite{EVANS_1}.

\subsection{Construction of orthogonal $n$-ary groupoids}

In the following example  sufficiently convenient and  general  way for the construction of systems of orthogonal $n$-ary groupoids is given.

\begin{example}
  Define operations $A_1(x_1, x_2, x_3)$, $A_2(x_1, x_2, x_3)$,  $A_3(x_1, x_2, x_3)$  over the set $M = \{ 0, \,  1, \, 2 \, \}$  in the following way. Take all $27$ triplets $ K = \{ (R_i,\, S_i, \, T_i) \, \mid \, R_i, S_i, T_i \in M,  i \in \overline{1, 27}\}$ in any fixed order and put
\begin{equation*}
\begin{split}
& A_1(0, 0, 0) = R_1, A_1(0, 0, 1) = R_2, A_1(0, 0, 2) = R_3, \dots,  A_1(2, 2, 2) = R_{27}, \\
& A_2(0, 0, 0) = S_1, A_2(0, 0, 1) = S_2, A_2(0, 0, 2) = S_3, \dots,  A_2(2, 2, 2) = S_{27}, \\
& A_3(0, 0, 0) = T_1, A_3(0, 0, 1) = T_2, A_3(0, 0, 2) = T_3, \dots,  A_3(2, 2, 2) = T_{27}.
\end{split}
\end{equation*}
The operations $A_1$, $A_2$ and $A_3$ form a   system of orthogonal operations. If we take this $27$ triplets in other order, then we obtain other system of orthogonal $3$-ary groupoids.
\end{example}

This way gives a possibility to construct easily  inverse system $B$ of orthogonal $n$-ary operations to a fixed system $A$ of orthogonal $n$-ary operations. Recall inverse system means that $B(A(x_1^n)) = x_1^n$, $x_i \in Q$.

\begin{example} \label{TERNARY_GROUPOIDS_}  \cite{PS_VS_11}.
We give example of three orthogonal ternary groupoids that are defined on four-element set $\{0,\,1,\, 2,\,3 \}$. Multiplication table of the first groupoid (in fact, of a quasigroup) is given in Example \ref{TERN_NO_REDUC_QUS}. Below we give multiplication tables of other two $3$-ary groupoids.

\[
\begin{array}{cccc}
\begin{array}{c|cccc}
B_0 & 0 & 1 & 2 & 3 \\
\hline
0  & 3 & 0 & 1 & 3 \\
1  & 0 & 2 & 3 & 0 \\
2  & 1 & 2 & 1 & 3 \\
3  & 1 & 1 & 2 & 2 \\
\end{array}
&
\begin{array}{c|cccc}
B_1 & 0 & 1 & 2 & 3 \\
\hline
0  & 2 & 1 & 1 & 0 \\
1  & 2 & 3 & 3 & 0 \\
2  & 0 & 2 & 1 & 3 \\
3  & 0 & 0 & 3 & 1 \\
\end{array}
&
\begin{array}{c|cccc}
B_2 & 0 & 1 & 2 & 3 \\
\hline
0  & 1 & 2 & 0 & 0 \\
1  & 2 & 0 & 3 & 1 \\
2  & 0 & 2 & 3 & 2 \\
3  & 3 & 2 & 1 & 1 \\
\end{array}
&
\begin{array}{c|cccc}
B_3 & 0 & 1 & 2 & 3 \\
\hline
0  & 3 & 3 & 2 & 2 \\
1  & 0 & 1 & 2 & 1 \\
2  & 0 & 2 & 0 & 3 \\
3  & 3 & 1 & 0 & 3 \\
\end{array}
\end{array}
\]
\[
\begin{array}{cccc}
\begin{array}{c|cccc}
C_0 & 0 & 1 & 2 & 3 \\
\hline
0  & 3 & 1 & 2 & 0 \\
1  & 2 & 1 & 1 & 2 \\
2  & 0 & 1 & 0 & 1 \\
3  & 3 & 1 & 2 & 3 \\
\end{array}
&
\begin{array}{c|cccc}
C_1 & 0 & 1 & 2 & 3 \\
\hline
0  & 1 & 2 & 1 & 3 \\
1  & 1 & 2 & 3 & 1 \\
2  & 0 & 2 & 2 & 0 \\
3  & 1 & 3 & 1 & 1 \\
\end{array}
&
\begin{array}{c|cccc}
C_2 & 0 & 1 & 2 & 3 \\
\hline
0  & 3 & 3 & 0 & 0 \\
1  & 2 & 1 & 0 & 1 \\
2  & 3 & 3 & 2 & 0 \\
3  & 3 & 0 & 2 & 3 \\
\end{array}
&
\begin{array}{c|cccc}
C_3 & 0 & 1 & 2 & 3 \\
\hline
0  & 2 & 1 & 0 & 0 \\
1  & 2 & 0 & 2 & 3 \\
2  & 3 & 3 & 2 & 0 \\
3  & 2 & 0 & 0 & 3 \\
\end{array}
\end{array}
\]
\end{example}

From formula $(q^n)!$  follows that there exist $(4^3)! = 64!$   orthogonal systems of $3$-ary groupoids over a set of order 4.

\subsection{Ciphers on base of orthogonal systems of $n$-ary operation}

Here we propose to use a system  of orthogonal $n$-ary  groupoids as additional procedure in order to construct almost-stream cipher \cite{SCERB_09_CSJM}.

 Orthogonal systems of $n$-ary quasigroups were studied in \cite{Stojakovic_86, Syrbu_90, DUD_SYRB}.
Such  systems  have  more uniform distribution of elements of base set and therefore such systems may be more preferable in protection against statistical cryptanalytic  attacks.

\begin{algorithm} \label{ALG3}  \cite{PS_VS_11}. Let $A$ be a non-empty finite alphabet, $k$ be a natural number, $x_1^t$ be a plaintext.
 Take  a system  of $n$ n-ary orthogonal operations $(A,f_i)$, $i = 1, 2, \dots, n$. This system defines a permutation $F$ of the set $A^n$.
  We propose  the following enciphering procedure.

\begin{itemize}
   \item Step 1:  $y_1^n = F^{\,l}(x_1^n)$, where $l\geq 1$, $l$ is a natural number, $l$ is vary from one  enciphering round to other. If $t<n$, then we can add to plaintext some "neutral" symbols.

  \item On the Steps $\geq 2$ it is possible to use  Feistel schema \cite{Feistel_73, MENEZES}. For example, we can do the following  enciphering procedure $z_1^n = F^s(y_2, y_{3}, \dots, y_{n}, x_{n+1})$, if arity $n\geq 2$, or $z_1^n = F^s(y_3, y_{4}, \dots, y_{n}, x_{n+1}, x_{n+2})$, if $n\geq 3$. And so on.
    \end{itemize}
The deciphering algorithm is based on the fact that  orthogonal system of n n-ary operations (\ref{SYSTEM_16})
  has a unique solution for any tuple of elements $a_1, \dots , a_n$.
 \end{algorithm}

Algorithm  \ref{ALG3} is sufficiently safe relative to  chosen ciphertext and plaintext attack since the key is a non-periodic sequence  of applications of permutation $F$, i.e. sequence of powers  of  permutation $F$.
Therefore any permutation of the group $\left< F \right>$  can be used by ciphering information using Algorithm \ref{ALG3}.

Recall  application of only one step Algorithm  \ref{ALG3} is not very safe since this procedure is not resistant relatively  chosen ciphertext attack and chosen plaintext attack.

\section{Combined algorithms}

\subsection{Modifications of Algorithm  \ref{ALG3}}

By our opinion some modifications of this algorithm are desirable.
Following "vector ideas"  \cite{MN_MP_09} we propose as the first step  to write any letter $u_i$ of a plaintext as $n$-tuple ($n$-vector) and after that to apply Algorithm  \ref{ALG3}. For example it is possible to use a binary representation of characters of the alphabet $A$.

It is possible to divide plain text $u_1, \dots, u_n$ on  parts and to use Algorithm  \ref{ALG3} to some parts, to a text a part of which has been ciphered by Algorithm \ref{ALG3} on a previous ciphering round.

It is possible to change in  Algorithm \ref{ALG3}  variables  $x_1, \dots, x_k$ $(1\leq k \leq (n-1))$ by some fixed elements of the set $Q$ and name these elements as leaders. Notice, if  $k=n-1$, then we obtain  $n$ chipering images from  any plaintext letter $u$.

If in  a system of orthogonal $n$-ary operations there is at least one $n$-ary quasigroup, then we can apply by ciphering of information Algorithm \ref{ALGn_ar} and Algorithm \ref{ALG3} together with some non-periodical frequency, i.e., for example, we can apply four times Algorithm \ref{ALGn_ar} and after this we can apply five times  Algorithm \ref{ALG3} and so on.

It is possible to  use as a period sequence decimal representation of an irrational or transcendent number.  In this case we can take as a key the sequence of application of Algorithm \ref{ALGn_ar} and Algorithm \ref{ALG3}.

 Proposed modifications  make  realization of chosen plaintext attack and chosen ciphertext attack more complicate.

Taking into consideration that in binary case one application of Algorithm \ref{ALG3} generates from one plaintext symbol $u$  two cipher symbols, say $v_1, v_2$, we may propose to apply Algorithm  \ref{ALG3} for two plaintext symbols (or to one cipher symbol and one plain symbol, else to two cipher symbols) simultaneously.

 We propose to use Algorithm \ref{ALGn_ar} and Algorithm \ref{ALG3} simultaneously.

\begin{algorithm} \label{ALG324} Suppose that we have a plaintext $x_1^t$, $t\geq n$.
\begin{enumerate}
\item Divide plaintext on $n$-tuples.
  \item We apply to any $n$-tuple  of plaintext $n$-ary permutation $F^{\, l}(x_1^n) = y_1^n$.
  \item To $n$-tuple $y_1^n$ we apply Algorithm \ref{ALGn_ar} (its binary or $k$-ary variant) $G(y_1^n) = z_1^n$. Probably it will be better, if $k<n$. \item We apply to  $n$-tuple $z_1^n$  $n$-ary permutation $F^s(z_1^n) = t_1^n$.
 \end{enumerate}
Deciphering algorithm is clear.
 \end{algorithm}

Below we denote the action of  the left (right, middle) translation  in the  power $a$ of a binary quasigroup $(Q, g_1)$   on the element $u_1$ by the symbol ${}_{g_1}T^a_{l_1}(u_1)$. And so on.

\begin{algorithm} \label{BIN_ALG4} Enciphering. Initially we have plaintext  $u_1, u_2, \dots, u_6$.
\begin{equation} \label{FORMULA_12}
\begin{split}
& Step \;1.  \\
& {}_{g_1}T^a_{l_1}(u_1) = v_1  \\
& {}_{g_2}T^b_{l_2}(u_2) = v_2\\
& F_1^c(v_1, v_2) = (v'_1, v'_2)\\
& Step \;2. \\
& {}_{g_3}T^d_{v'_1}(u_3) = v_3\\
& {}_{g_4}T^e_{v'_2}(u_4) = v_4\\
& F_2^f(v_3, v_4) = (v'_3, v'_4)\\
& Step \;3.\\
& {}_{g_5}T^g_{v'_3}(u_5) = v_5\\
& {}_{g_6}T^h_{v'_4}(u_6) = v_6\\
& F_3^i(v_5, v_6) = (v'_5, v'_6)
\end{split}
\end{equation}
And so on.
We obtain ciphertext $v'_1, v'_2, \dots, v'_6$.

Deciphering. Initially we have ciphertext $v'_1, v'_2, \dots, v'_6$.
\begin{equation}\label{FORMULA_212}
\begin{split}
& Step \;1. \\
& F_1^{-c}(v'_1, v'_2) = (v_1, v_2)\\
& {}_{g_1}T^{-a}_{l_1}(v_1) = u_1\\
& {}_{g_2}T^{-b}_{l_2}(v_2)  = u_2\\
& Step \;2. \\
& F_2^{-f}(v'_3, v'_4) = (v_3, v_4)\\
& {}_{g_3}T^{-d}_{v'_1}(v_3) = u_3\\
& {}_{g_4}T^{-e}_{v'_2}(v_4) = u_4\\
& Step \;3. \\
& F_3^{-i}(v'_5, v'_6) = (v_5, v_6)\\
& {}_{g_5}T^{-g}_{v'_3}(v_5) = u_5\\
& {}_{g_6}T^{-h}_{v'_4}(v_6) = u_6\\
\end{split}
\end{equation}
We obtain plaintext  $u_1, u_2, \dots, u_6$.
\end{algorithm}

It is clear that Algorithm \ref{ALG1_POWER_OF_TRANSL}  is a partial case of Algorithm \ref{BIN_ALG4}.

As in Algorithm \ref{ALG1_POWER_OF_TRANSL},  in Algorithm \ref{BIN_ALG4} the elements $a, b, c, \dots, h$ should be vary in order to protect this algorithm  against chosen  plain-text and chosen cipher-text attack.

   Algorithm \ref{BIN_ALG4} allows  to obtain almost "natural" stream cipher, i.e. stream cipher that encode a pair of elements of a plaintext on any step.  It is easy to see that Algorithm \ref{BIN_ALG4} can be generalized on $n$-ary ($n\geq 3$) case. One of the possible generalizations is realized in Algorithm  \ref{ALG3124}.

Additional researches are necessary for the proposed in this subsection modifications.

\subsection{Stream cipher on base of orthogonal system of binary parastrophic quasigroups}

This subsection is more of algebraic than cryptographical character. For the construction of Algorithms \ref{ALGn_ar} and  \ref{ALG3} we propose the use  of  orthogonal systems of binary parastrophic quasigroups.

We start from the following theorem \cite{MS05}. Here expression
$A\bot {{}^{(23)}A}$ means that  quasigroups $(Q, A)$ and $(Q, {{}^{(23)}A})$ are orthogonal.

\begin{theorem} \label{TH2} For a finite  quasigroup $(Q, A)$ the following equivalences are
fulfilled:

(i)    $ A\bot {{}^{(12)}A}\Longleftrightarrow ((x \backslash  z) \cdot x  = (y \backslash  z)
\cdot y \Longrightarrow x=y)$;

(ii)    $A\bot {{}^{(13)}A} \Longleftrightarrow (zx\cdot x = zy\cdot y \Longrightarrow x=y)
$;

(iii) $ A \bot {{}^{(23)}A} \Longleftrightarrow (x\cdot xz = y\cdot yz \Longrightarrow x=y)
$;

(iv)   $A \bot {{}^{(123)}A} \Longleftrightarrow (x\cdot zx  = y\cdot zy   \Longrightarrow x=y)$;

(v)    $A  \bot {{}^{(132)}A}  \Longleftrightarrow (x z\cdot x
  = yz\cdot y   \Longrightarrow x=y)$

 for all $x, y, z \in Q$.
\end{theorem}

In order to construct quasigroups mentioned in Theorem \ref{TH2}
probably computer search is preferable. It is possible to use   GAP and   Prover \cite{MAC_CUNE_PROV}.

\begin{definition}
A  $T$-quasigroup $(Q,A)$ is a quasigroup of the form $A(x, y) = \varphi x + \psi y + c$, where  $(Q, +)$ is an abelian group, $\varphi, \psi$ are some fixed automorphisms of this group, $c$ is a fixed element of the set $Q$ \cite{pntk, tkpn}.
\end{definition}

If $(Q,\cdot)$ is a $T$-quasigroup of  the form $x\cdot y = \varphi  x + \psi y + c$, then  its parastrophes
have the following forms, respectively:
\begin{equation} \label{PARAST_OF_T_QUAS}
\begin{split}
& x \overset{(12)}{\cdot} y = \psi x + \varphi  y   + c,\\
& x \overset{(13)}{\cdot} y =  \varphi^{-1} x - \varphi^{-1} \psi y - \varphi^{-1} c,\\
& x \overset{(23)}{\cdot} y = - \psi^{-1} \varphi x + \psi^{-1} y - \psi^{-1} c,\\
& x \overset{(123)}{\cdot} y =  - \varphi^{-1} \psi x + \varphi^{-1} y - \varphi^{-1} c,\\
& x \overset{(132)}{\cdot} y =  \psi^{-1} x - \psi^{-1} \varphi y  - \psi^{-1} c.
\end{split}
\end{equation}
See, for example, \cite{MS05}.

In order to construct a quasigroup $(Q, A)$ that is orthogonal with its parastrophe in more theoretical way it is possible to use the following theorem \cite{MS05}.

\begin{theorem} \label{T1}
For a $T$-quasigroup $(Q,A)$ of the form $A(x, y) =  \varphi x + \psi y + c$ over an abelian group $(Q, +)$ the
following equivalences are fulfilled:

(i) $A \bot {}^{12}A \Longleftrightarrow (\varphi - \psi), (\varphi + \psi)$ are permutations of the set $Q$;

(ii) $A\bot {}^{13}A \Longleftrightarrow (\varepsilon + \varphi)$ is a permutation of the set $Q$;

(iii) $A\bot {}^{23}A \Longleftrightarrow (\varepsilon + \psi)$ is a permutation of the set $Q$;

(iv) $A\bot {}^{123}A \Longleftrightarrow (\varphi+ \psi^2)$ is a permutation of the set $Q$;

(v) $A\bot {}^{132}A \Longleftrightarrow (\varphi^2 + \psi)$ is a permutation of the set $Q$.
\end{theorem}

\begin{corollary}\label{PAR_ORTH_CYCL}
$T$-quasigroup $(Z_{p}, \circ)$ of the form $x\circ y = k \cdot x + m \cdot y + c$, where $(Z_{p}, +)$ is the cyclic group of a prime order $p$, $k, m, c \in Z_p$; $k, m, k+m, k-m,  k+1, m+1, k^2+m, k+m^2 \neq 0 \pmod{p}$,  where the operation $\cdot$ is multiplication modulo  $p$, is orthogonal to any of its parastrophes.
\end{corollary}
 Quasigroups  from Corollary \ref{PAR_ORTH_CYCL} are  suitable objects to construct  above mentioned Algorithms (binary case).

The following table contains connections between different kinds of translations in different parastrophes of a binary quasigroup $(Q,\cdot)$ \cite{SCERB_03, SCERB_07}.

\hfill Table 1. \label{Table_0}
\[
\begin{array}{|c||c| c| c| c| c| c|}
\hline
  & \varepsilon  & (12) & (13) & (23) & (123) & (132)\\
\hline\hline
R  & R & L & R^{-1} & P & P^{-1} & L^{-1} \\
\hline
L  & L & R & P^{-1} & L^{-1} & R^{-1} & P \\
\hline
P  & P & P^{-1} & L^{-1} & R & L & R^{-1} \\
\hline
R^{-1} & R^{-1} & L^{-1} & R & P^{-1} &  P & L \\
\hline
L^{-1}  & L^{-1} & R^{-1} & P & L & R & P^{-1} \\
\hline
P^{-1}  & P^{-1} & P & L & R^{-1} & L^{-1} & R \\
\hline
\end{array}
\]

From Table 1 it follows,  for example, that $R^{(13)} = R^{-1}$.

\subsection{T-quasigroup based stream code}

We give a numerical example   of encryption Algorithm \ref{BIN_ALG4} based on $T$-quasigroups. Notice the number $257$ is prime.

\begin{example} \label{ALG0123} Take the cyclic group $(Z_{257}, +)= (A, +)$.
\begin{enumerate}
  \item Define  T-quasigroup $(A, \ast)$ with the form $x\ast y = 2\cdot x + 131\cdot y + 3$ with a leader element $l_1$, say,  $l_1=17$. Denote the mapping $x\mapsto x\ast l_1$ by the letter $R_{l_1}$, i.e. $R_{l_1}(x) = x\ast l_1$ for all $x\in A$.

In order to find the mapping $R_{l_1}^{-1}$ taking into consideration Table 1  we find the form of operation $ \overset{(13)}{\ast}$ using formula (\ref{PARAST_OF_T_QUAS}).   We have    $x \overset{(13)}{\ast} y =  129 \cdot  x + 63 \cdot y + 127$, $R_{l_1}^{-1}x = x \overset{(13)}{\ast} l_1 = R_{l_1}^{(1 3)}x$.

In some sense quasigroup $(A, \overset{(13)}{\ast})$ is the "right inverse quasigroup" to quasigroup  $(A, \ast)$. From identity
(\ref{(4ec)})  follows that quasigroup $(A, \overset{(132)}{\ast})$ is the "left inverse" quasigroup to quasigroup  $(A, \ast)$.
       Notice from Corollary \ref{PAR_ORTH_CYCL}  follows that $(A, \ast) \bot (A, \overset{(13)}{\ast})$.
\item
Define  T-quasigroup $(A, \circ)$ with the form $x\circ y = 10\cdot x + 81\cdot y + 53$ with a leader element $l_2$, say,  $l_2 = 71$. Denote the mapping $x\mapsto l_2\ast x$ by the letter $L_{l_2}$, i.e. $L_{l_2}(x) = l_2\circ x$ for all $x\in A$.

 In order to find the mapping $L^{-1}_{l_2}$  we use Table 1 and find the form of operation $\overset{(23)}{\circ}$ by  formula (\ref{PARAST_OF_T_QUAS}). We have  $x \overset{(23)}{\circ} y = 149\cdot  x + 165\cdot y +250$.

  \item Define a system  of two parastroph orthogonal T-quasigroups $(A, \cdot)$ and  $(A, \overset{(23)}{\cdot})$  in the following way
  \[\left\{
\begin{array}{ll}
 x\cdot y = 3\cdot x + 5\cdot y + 6\\
 x \overset{(23)}{\cdot} y = 205 \cdot  x + 103 \cdot y + 153
\end{array}
\right.
\]

Denote quasigroup system  $(A, \cdot, \overset{(23)}{\cdot})$ by $F(x,y)$, since  this system  is a function of two variables.

In order to find the mapping $F^{-1}(x,y)$ we solve the system of linear equations
\[\left\{
\begin{array}{ll}
3\cdot x + 5\cdot y + 6 & = a \\
205 \cdot  x + 103\cdot y + 153 & = b
\end{array}
\right.
\]

We have $\Delta = 55$, $1/\Delta = 243$, $x = 100\cdot a + 70\cdot b + 255$, $y = 43\cdot a + 215\cdot b$.
Therefore we have, if $F(x, y) = (a, b)$, then
 $F^{-1}(a, b) = (100\cdot a + 70\cdot b + 255, 43\cdot a + 215\cdot b) $, i.e.
\[\left\{
\begin{array}{ll}
x = 100\cdot a + 70\cdot b + 255\\
y = 43\cdot a + 215\cdot b
\end{array}
\right.
\]
\end{enumerate}

We have defined the mappings $g_1 = R_{l_1}$, $g_2 = L_{l_2}$, $F$ and now we can use  them in Algorithm \ref{BIN_ALG4}.

Let $212;\, 17;\, 65;\, 117$ be a plaintext. We take the following values in formula~(\ref{FORMULA_12}): $a = b = d = e = f = 1; c = 2$. Below
we use Gothic font to distinguish leader elements, i.e.  $\frak{17}, \frak{71}$ are  leader elements.
Then

\medskip

Step 1.

$g_1(212) = 212\ast \frak{17} = 2 \cdot 212 + 131 \cdot 17 + 3 = 84$

$g_2(17) = \frak{71}\circ 17 = 10 \cdot 71 + 81 \cdot 17 + 53 = 84$

$F(84;84) = (3\cdot 84 + 5\cdot 84 + 6; 205\cdot 84 + 103\cdot 84 + 153) = (164; 68)$

$F(164;68) = (3\cdot 164 + 5\cdot 68 + 6; 205\cdot 164 + 103\cdot 68 + 153) = \textbf{(67; 171)}$

\medskip

Step 2.

$g_1(65) = 65\ast 67 = 2 \cdot 65 + 131 \cdot 67 + 3 = 172$

$g_2(117) = 171\circ 117 = 10 \cdot 171 + 81 \cdot 117 + 53 = 189$

$F(172;189) = (3\cdot 172 + 5\cdot 189 + 6; 205\cdot 172 + 103\cdot 189 + 153) = \textbf{(182; 139)}$

We obtain the following ciphertext $67;\, 171;\, 182;\, 139$.

\medskip

For deciphering we  use   formula~(\ref{FORMULA_212}).

\medskip

Step 1.

$F^{-1}(67; 171) = (100\cdot 67 + 70\cdot 171 + 255, 43\cdot 67 + 215\cdot 171) = (164; 68)$

$F^{-1}(164; 68) = (100\cdot 164 + 70\cdot 68 + 255, 43\cdot 164 + 215\cdot 68) = (84; 84)$

$g_1^{-1}(84) = 84\overset{(13)}{\ast} 17 = 129 \cdot 84 + 63 \cdot 17 + 127 = \textbf{212}$

$g_2^{-1}(84) = 71\overset{(23)}{\circ} 84 = 149\cdot 71 + 165\cdot 84 + 250 = \textbf{17}$

\medskip

Step 2.

$F^{-1}(182; 139) = (100\cdot 182 + 70\cdot 139 + 255, 43\cdot 182 + 215\cdot 139) = (172; 189)$

$g_1^{-1}(172) = 172 \overset{(13)}{\ast} 67 = 129 \cdot 172 + 63 \cdot 67 + 127 = \textbf{65}$

$g_2^{-1}(189) = 171\overset{(23)}{\circ} 189 = 149\cdot 171 + 165\cdot 189 + 250 = \textbf{117}$
\end{example}

A little program using freeware version of programming language Pascal was developed. First little experiments demonstrate that encoding-decoding is executed sufficiently fast.\footnote{The author thanks D.I. Pushkashu and A.V. Shcherbacov for their  help by the writing of this program.}

\begin{remark}
Proper binary groupoids are more preferable than linear quasigroups by construction of the mapping $F(x,y)$ in order to make  encryption more safe, but in this case decryption may be slower than in linear quasigroup case and definition of these groupoids needs  more computer (or some other device) memory. The same remark is true for the choice of the function $g$. Maybe a golden mean in this choice problem is to use linear quasigroups over non-abelian, especially simple,  groups.
\end{remark}

\begin{remark}
In this cipher there exists a  possibility of protection against standard statistical attack. For this scope it is possible to denote more often used letters or pair of letters by more than one integer or by more than one  pair of integers.
\end{remark}

\subsection{Some generalization of functions of Algorithm \ref{BIN_ALG4}}

We give a method for the construction of  functions that it is possible to use in  cryptographical procedures. Suppose that all functions are defined on a set $Q$.
Functions $F(x_1^n)$ and $g(x_1^n)$ are functions of  $n$ variables.

Function $F$ ($n$ orthogonal groupoids, a permutation of the set $Q^n$) has  inverse function of $n$ variables $F^{-1}(x_1^n)$ such that $F(F^{-1}(x_1^n)) = F^{-1}(F(x_1^n)) = x_1^n$.

\iffalse
Recall function $g$ (binary quasigroup) has four "local" inverse functions
\[
\begin{array}{ll}
g^{-1}_1(g(x, y), y) = x & (identity \;(\ref{(4e)}))\\
g^{-1}_2(x, g(x, y)) = y & (identity \;(\ref{(3e)}))\\
g^{-1}_3(g(x, y), x) = y & (identity \;(\ref{(3ec)}))\\
g^{-1}_4(y, g(x, y)) = x & (identity \;(\ref{(4ec)}))
\end{array}
\]
\fi

We recall, if $g$ is $n$-ary quasigroup operation, then, in general, we cannot decode values $x, y$, for example, from equality $g(\overline{a}^{\,n-2}, x, y) = b$,  but we can easy solve equation $g(\overline{a}^{\, n-1}, x) = b$ of one variable, i. e. we can decode value of variable $x$.

Taking into consideration this quasigroup feature, we describe the set (clone) of functions that it is possible to use in cryptology on base of these two kinds of functions, namely,  functions $F$ and $g$.  We shall use concept of term \cite{WIKI_7} to define cryptographical terms (cryptographical functions) inductively.

Cryptographical function (cryptographical term) below in Case 3 means that encoding and decoding of a text using this  function  (this term) is performed uniquely.

\begin{algorithm} \label{INDUCTION_CRY_TERM}
\begin{enumerate}
  \item Any individual constant is a cryptographical term.
  \item Any individual variable  is a cryptographical term.
  \item
  \begin{enumerate}
    \item If $g$ is an $n$-ary quasigroup functional constant ($(Q, g)$ is an $n$-ary quasigroup) and $t$ is a term, $b_{\,1}^{\,n}$ are individual constant,  then $g^{\,a} (b_1^{i-1}, t, b_{i+1}^n)$, $i\in \overline{1,n}$, where $a \in \mathbb{Z}$, is  a cryptographical term.
      \item   If $F$ is a permutation of a set $Q^{\,n}$ which is constructed using $n$ orthogonal $n$-ary groupoids and $t_1, t_2, \dots, t_n$ are quasigroup cryptographical terms, then $F^a(t_1, \dots, $ $ t_n)$, where $a \in \mathbb{Z}$, is a cryptographical term.
\end{enumerate}
\end{enumerate}
\end{algorithm}

\begin{example} \label{Rule_3_a}
Let $Q = B\times B$ be a non-empty set, $F$ be a pair of orthogonal groupoids every of which is defined on the set $B$, and $(Q, g)$ be a ternary quasigroup.
Then $g(q_1, q_2, F)$, where $q_1, q_2$ are fixed elements of $Q$,  is  a cryptographical term constructed following Rule 3, (a) of Algorithm   \ref{INDUCTION_CRY_TERM}.
\end{example}

In Example \ref{ALG0123}  cryptographical term $F^{\,a}(g_1, g_2)$ is constructed following Rule 3, (b) of  Al\-go\-rithm \ref{INDUCTION_CRY_TERM}.
Indeed, the function $F$ is a pair of parastrophic orthogonal $T$-quasigroups that are defined on  the set $Z_{257}$, i.e.  $F$ is a permutation of the set $Z_{257}\times Z_{257}$;  $(Z_{257}, g_1)$, $(Z_{257}, g_2)$ are binary $T$-quasigroups, and  $a = 1; 2$.

\begin{algorithm} \label{ALG3124} Suppose that we have $n$-ary permutation $F$,  $n$ procedures $G_j$ (they may be of various arity and it is supposed that leader elements are used) and plaintext $x_1^t$.

By the letter $y$ with an index we denote an element of enciphered text or a leader element.
We propose the following enciphering procedure.

The $i$-th step of this procedure can have the following form
\begin{equation} \label{FORMUL_PR_3}
{}_iF^k(G_1(y_1^m, x_i), \dots, G_n(y_1^r, x_i)) = {}_iy_1^n
\end{equation}
Deciphering algorithm is executed "from the top to the bottom" in general and "from the bottom  to the top" on any step. See more details in Algorithm \ref{BIN_ALG4}.
 \end{algorithm}

\subsection{On quasigroup based cryptcode}

Using possibilities that give us Algorithms \ref{INDUCTION_CRY_TERM} and \ref{ALG3124} we give an example  of  a quasigroup based hybrid\footnote{Hybrid  idea is sufficiently known. For example, see  \cite[page 2]{SCERB_03_1}, \cite[page 65]{SCERB_03}.} of a code and a cypher. Following Markovski, Gligoroski, and Kocarev  \cite{MARK_GLIG_KOC_08, MARK_GLIG_KOC_07},  we  name such hybrid  as a cryptcode.

We shall use Klein group  $Z_2 \oplus Z_2$, its automorphism group and the system of three ternary orthogonal groupoids (Example \ref{TERNARY_GROUPOIDS_}).

Denote   elements of the group $Z_2 \oplus Z_2$ as follows: $\{(0;0),$ $(1;0),$ $ (0;1), $ $(1;1)\}$.  The group $Aut(Z_2 \oplus Z_2)$ consists of the following automorphisms :
\[
 \left(\begin{array}{cc} 1 & 0\\ 0 &1 \end{array}\right), \left(\begin{array}{cc} 1 & 0\\ 1 &1 \end{array}\right), \left(\begin{array}{cc} 1 & 1\\ 0 &1 \end{array}\right), \left(\begin{array}{cc} 0 & 1\\ 1 &0 \end{array}\right), \left(\begin{array}{cc} 1 & 1\\ 1 &0 \end{array}\right), \left(\begin{array}{cc} 0 & 1\\ 1 &1 \end{array}\right)
\]
Denote these automorphisms by the letters $\varepsilon, \varphi_2, \varphi_3, \varphi_4, \varphi_5, \varphi_6$, respectively.

\smallskip
Notice  $\varphi_2^{2} =  \varphi_3^{2} = \varphi_4^{2} = \varepsilon, \varphi_5^{2} = \varphi_6, \varphi_6^{2} = \varphi_5$.
It is known that $Aut(Z_2 \oplus Z_2) \cong S_3$ \cite{HALL, KM}.

For convenience we give Cayley table of the group  $Aut(Z_2 \oplus Z_2)$.
\[
\begin{array}{c|cccccc}
\cdot & \varepsilon & \varphi_2 & \varphi_3 & \varphi_4 & \varphi_5 & \varphi_6\\
\hline
\varepsilon & \varepsilon & \varphi_2 & \varphi_3 & \varphi_4 & \varphi_5 & \varphi_6\\
\varphi_2 &  \varphi_2 & \varepsilon & \varphi_5 & \varphi_6 & \varphi_3 & \varphi_4 \\
\varphi_3 &  \varphi_3 & \varphi_6 & \varepsilon & \varphi_5 & \varphi_4 & \varphi_2 \\
\varphi_4 &  \varphi_4 & \varphi_5 & \varphi_6 & \varepsilon & \varphi_2 & \varphi_3 \\
\varphi_5 &  \varphi_5 & \varphi_4 & \varphi_2 & \varphi_3  & \varphi_6 & \varepsilon \\
\varphi_6 &  \varphi_6 & \varphi_3 & \varphi_4 & \varphi_2  & \varepsilon & \varphi_5 \\
\end{array}
\]

Information on codes is in \cite{BLAHUT}. We shall use  a  code that is given in  \cite[Example 19]{MS1}. Suppose that the symbols $x, y$ are informational symbols and the symbol $z$ is a check symbol. Remember, $x, y, z \in (Z_2 \oplus Z_2)$. We propose the following check equation $x + \varphi_5 y + \varphi_6 z = (0;0)$, i.e., we propose   the following formula to find the element $z$:
\begin{equation} \label{CHECK_FORMULA}
z = \varphi_5 x + \varphi_6 y
\end{equation}
Recall, statistical investigations of J. Verhoeff \cite{Ver69} and D.F. Beckley \cite{Bec67} have shown
that the most frequent errors made by human operators during transmission of
data are single errors (i.e. errors in exactly one component), adjacent transpositions
(in other words errors made by interchanging adjacent digits, i.e. errors of
the form $ab \rightarrow ba $), and insertion or deletion errors. We note, if all codewords
are of equal length, insertion and deletion errors can be detected easily.

Proposed code detects any single, transposition, and twin ($aa \rightarrow bb$) errors \cite{MS1}.

\medskip

Further we construct three $T$-quasigroups over the group $Z_2 \oplus Z_2$:

$(Z_2 \oplus Z_2, D)$ with the form $D(x, y) = \varphi_3 x + \varphi_6 y + a_1$;

$(Z_2 \oplus Z_2, E)$ with the form $E(x, y) = \varphi_2 x + \varphi_5 y + a_2$;

$(Z_2 \oplus Z_2, F)$ with the form $F(x, y) = \varphi_3 x + \varphi_5 y + a_3$.

\medskip

We  use the following
\begin{theorem}\label{T_ORTHOGON} A $T$-quasigroup $(Q,\cdot)$ of the form $x\cdot y = \alpha  x + \beta y + c$
 and a $T$-quasigroup  $(Q,\circ)$ of the form $x\circ y = \gamma x + \delta y + d$, both over a
 group $(Q,+)$ are orthogonal if and only if the map $\alpha^{-1}\beta - \gamma^{-1}\delta$ is an automorphism
of the group $(Q,+)$ \cite{MS05}.
\end{theorem}

\begin{lemma}
The quasigroups $(Z_2 \oplus Z_2, D)$, $(Z_2 \oplus Z_2, E)$, and
$(Z_2 \oplus Z_2, F)$ are orthogonal in pairs.
\end{lemma}
\begin{proof}
We can use  Theorem \ref{T_ORTHOGON} and Cayley table of the group $Aut(Z_2 \oplus Z_2)$.
\end{proof}

Define three ternary operations in the following way: $K_1(D(x, y), z) = D(x, y) + z$, $K_2(E(x, y), z) = E(x, y) + z$,
$K_3(F(x, y), z) = F(x, y) + z$.

\begin{lemma}
The triple of ternary operations $K_1(x, y, z),  K_2(x, y, z), K_3(x, y, z)$  forms  orthogonal system of operation.
\end{lemma}
\begin{proof}
We solve the following system  of equations
\begin{equation} \label{TT_SYSTEM}
\left\{
\begin{split}
& \varphi_3 x + \varphi_6 y + a_1 + z  = b_1 \\
& \varphi_2 x + \varphi_5 y + a_2 + z  = b_2 \\
& \varphi_3 x + \varphi_5 y + a_3 + z = b_3  \\
\end{split}
\right.
\end{equation}
where $b_1, b_2, b_3$ are fixed elements of the set $Z_2 \oplus Z_2$.

We use  properties of the groups $(Z_2 \oplus Z_2)$ and  $Aut(Z_2 \oplus Z_2)$.
\begin{equation} \label{TT_SYSTEM_3}
\left\{
\begin{split}
& \varphi_3 x + \varphi_6 y + z = b_1 +  a_1 \\
& \varphi_2 x + \varphi_5 y + z = b_2 +  a_2 \\
& \varphi_3 x + \varphi_5 y + z = b_3 +  a_3 \\
\end{split}
\right.
\end{equation}

We are doing  the  following transformations of the system  (\ref{TT_SYSTEM_3}): (first  row + third  row) $\rightarrow$ first row; (second row + third  row) $\rightarrow$ second row; and obtain the following system:
\begin{equation} \label{TT_SYSTEM_1}
\left\{
\begin{split}
& y = b_1 +  a_1 + b_3 + a_3 \\
& x =  \varphi_4(b_2 +  a_2 + b_3 + b_4) \\
& \varphi_3 x + \varphi_5 y + z = b_3 +  a_3 \\
\end{split}
\right.
\end{equation}

If in the system (\ref{TT_SYSTEM_1})  in the third equation we replace $x$ by $\varphi_4(b_2 +  a_2 + b_3 + b_4)$ and $y$ by $b_1 +  a_1 + b_3 + a_3$,  then we obtain
\begin{equation} \label{TtT_SYSTEM}
\left\{
\begin{split}
& x =  \varphi_4(b_2 +  a_2 + b_3 + a_3) \\
& y = b_1 +  a_1 + b_3 + a_3 \\
& z =  b_3 +  a_3 + \varphi_5 (b_1 + a_1 + b_2 + a_2) \\
\end{split}
\right.
\end{equation}

Therefore  the system (\ref{TT_SYSTEM})  has a unique solution for any fixed elements $b_1, b_2, b_3\in (Z_2 \oplus Z_2)$, operations $K_1(x, y, z),  K_2(x, y, z), K_3(x, y, z)$ are orthogonal.
\end{proof}

Triple of orthogonal operations $K_1(x, y, z),  K_2(x, y, z), K_3(x, y, z)$ defines on the set $Q^{\,3}$ a permutation.
Denote this permutation  by the letter $K$.

We shall use the system of three ternary orthogonal groupoids  $(Q, A)$, $(Q, B)$,  $(Q, C)$ of order 4 from Example \ref{TERNARY_GROUPOIDS_}. See also  \cite{PS_VS_11}. Denote permutation that defines this  system of three ternary orthogonal groupoids  by the letter $M$.

In order to use the system of orthogonal groupoids and the system of orthogonal $T$-quasigroups simultaneously we redefine the basic set of the $T$-quasigroups in the following (non-unique) way $(0;0) \rightarrow 0$, $(1;0) \rightarrow 1$, $(0;1) \rightarrow 2$, $(1;1) \rightarrow 3$.

We propose the following cryptographical term (a cryptographical  primitive):
\begin{equation*}
H(x, y, z)  = M^k(K^l(x,y, z)),  k, l \in \mathbb{Z}
\end{equation*}

Transformation $H$ is a permutation of the set $Q^{\,3}$. Indeed, this transformation is a composition of two permutations: $K^{\,l}$ and
 $M^{k}$.

Therefore we propose the following
\begin{algorithm} \label{CODE_CIPHER_PROCEDURE}
\begin{enumerate}
  \item  Take a pair of information symbols $a, b \in (Z_2 \oplus Z_2)$;
  \item by formula (\ref{CHECK_FORMULA}) find value of the check symbol $c$;
  \item apply cryptographical term $H$ to the triple $(a, b, c)$;
  \item  take a pair of information symbols $d, e \in (Z_2 \oplus Z_2)$;
  \item by formula (\ref{CHECK_FORMULA}) find value of the check symbol $f$;
  \item change values of the numbers $k, l$ in the cryptographical term $H$; also it is possible to change the term $H$ by some other term of such or other type;
  \item apply cryptographical term $H$ to the triple $(d, e, f)$;
  \item and so on.
  \end{enumerate}
\end{algorithm}

Procedure of decoding in Algorithm \ref{CODE_CIPHER_PROCEDURE} is clear.

Recall, the number $N(n)$  of mutually (in pairs) orthogonal Latin squares of order $n$ fulfills the following inequality  $N(n) \leq (n-1)$ \cite{LM}. Then for $n=4$ we have $N(4) \leq 3$.
Therefore,  for real applications an analog of Algorithm  \ref{CODE_CIPHER_PROCEDURE} should be constructed over a set of order more than 4 and, probably, with more powerful code \cite{BAK_ILEV_09}.

\subsection{A comparison of the "power" of proposed algorithms}

We shall compare how many permutations and  of what length can be generated and can be used by the working of  some above mentioned algorithms.

Algorithm  \ref{ALG1}.
If we shall use only one quasigroup $(Q, \cdot)$, $|Q| = n$,   then we can obtain  by encoding not more than  $n$ permutations of the group $S_n$.

Algorithm \ref{ALG1_POWER_OF_TRANSL}.
If we shall use only one quasigroup $(Q, \cdot)$, $|Q| = n$,    then we shall use by encoding  the set $S = \cup_{i=1}^n \left< L_{a_i} \right>$ of permutations which is a subset of the left  multiplication  group $LM$  of quasigroup $(Q, \cdot)$.  We recall $LM(Q, \cdot) = \left< L_x \, | \, x \in Q \right> $ \cite{VD, HOP, SCERB_03}.

It is possible to construct a quasigroup $(Q, \cdot)$ such that $LM(Q, \cdot) = S_Q$. Notice, it is proved \cite{DR_KEPKA} that there exist quasigroups with the property $LM(Q, \cdot) = A_Q$, where $A_O$ is the alternating group defined on the set $Q$ \cite{KM, HALL}.

Therefore by encoding using Algorithm \ref{ALG1_POWER_OF_TRANSL} we can obtain not more than $|S_n| = n!$ permutations.

Situation with  Algorithm \ref{ALGn_ar} is similar to the  situation with   Algorithm  \ref{ALG1}. Since by encoding  translations of an $m$-ary quasigroup $(Q, f)$  are used,  we can obtain not more than $|S_n| = n!$ permutations. The properties of multiplication group (more exactly, multiplication groups) of $n$-ary quasigroups are not researched well.

Information on the multiplication groups of linear $n$-ary quasigroups is in \cite{MARS}. These quasigroups  are used  in  \cite{Petrescu_07, Petrescu_10} by construction of some ciphers (see above).

Algorithm \ref{ALG1_POWER_OF_TRANSL_N_ARY} is a synthesis of Algorithms \ref{ALG1_POWER_OF_TRANSL} and \ref{ALG1}.  Here by the symbol $T_i$ we denote translations of an $n$-ary quasigroup $(Q, f)$. It is clear that the order of the set $S = \cup \left< T_i \right>$ can be large  but cannot be more than $|S_n| = n!$.

In Algorithm \ref{ALG3}  elements of the cyclic group $\left< F \right> \subset S_{n^m}$, where $|Q| = n$, $m$ is the arity of  orthogonal groupoids, can  appear.  In the above-mentioned   inclusion  cannot be equality even theoretically, since the minimal number of generators of the symmetric group is equal to two \cite{KM, HALL}.

It is well known that a cycle of order $n$ and a cycle of order two generate the symmetric  group $S_n$ \cite{KM, HALL}.

The group $S_{n^m}$ is an upper bound of the sets of permutations that can be generated during the work of Algorithms \ref{ALG324}, \ref{ALG3124}.
For  Algorithm \ref{BIN_ALG4} the group $S_{n^2}$ is such upper bound. It is clear that in Algorithm \ref{BIN_ALG4}  by the encryption any permutation of the group $S_{n^2}$ may be realized. But it also is clear that this is not necessary from the cryptographical point of view.

The possible number of  permutation generated during the work of the algorithm from  Example \ref{ALG0123} is bounded by the number $(257^2)! = 66049!$ and during the work of Algorithm \ref{CODE_CIPHER_PROCEDURE} is bounded by the number $(64)!$.

\medskip

\noindent \textbf{Acknowledgement.}
The author started this project together with  Professor Piroska  Cs\"org\"o \cite{PS_VS_11}. Unfortunately Prof. Cs\"org\"o has informed the author that she cannot continue this project. The author is grateful to Prof. Cs\"org\"o for useful discussions and the help by writing this paper.

\noindent \footnotesize
{Institute of Mathematics and \\
Computer Science \\
Academy of Sciences of Moldova  \\
Academiei str. 5,  MD$-$2028 Chi\c{s}in\u{a}u  \\
Moldova  \\
E-mail: \emph{scerb@math.md }}

\end{document}